\documentclass[11pt]{article}
\usepackage{amssymb,amsmath,amsthm,amsfonts, commath, mathrsfs, mathtools, secdot}
\usepackage[T1]{fontenc}		
\usepackage[a4paper, margin=0.97in]{geometry}

\usepackage{bbm}
\usepackage{extarrows}
\usepackage{enumerate}
\usepackage{multicol}	
\usepackage{fge}
\usepackage{oldgerm}
\usepackage{hyperref} 
\usepackage{paralist} 
\usepackage{accents}
\usepackage{pifont}
\usepackage{authblk}

\makeatletter
\newcommand{\myitem}[1]{%
\item[#1]\protected@edef\@currentlabel{#1}%
}
\makeatother

\usepackage{tikz}
\usetikzlibrary{datavisualization}
\usetikzlibrary{datavisualization.formats.functions}
\usepackage{pgfplots}
\usepackage{tkz-fct}
\pgfplotsset{compat=1.18}
\pgfplotsset{every axis/.append style={
    axis x line=middle,    
    axis y line=middle,    
    axis line style={<->}, 
    xlabel={$x$},          
    ylabel={$y$},          
    },
    cmhplot/.style={color=blue,mark=none,line width=1pt,<->},
    soldot/.style={color=blue,only marks,mark=*},
    holdot/.style={color=blue,fill=white,only marks,mark=*},
}
    
\tikzset{>=stealth}

\usepackage{blindtext}
\usepackage{caption}
\usepackage[natbib=true, backend=biber, style=numeric]{biblatex}
\usepackage{csquotes}
\sectiondot{subsection}
\usepackage{xcolor}

\usepackage{color}

\newtheoremstyle{Assump}%
  {3pt}
  {3pt}
  {\itshape}
  {}
  {\bfseries}
  {.}
  {.5em}
  {\thmname{#1} \thmnumber{#2} \thmnote{\normalfont#3}}

\newtheorem{theorem}{Theorem}[section]
\newtheorem{lem}[theorem]{Lemma}

\newtheorem{cor}[theorem]{Corollary}
\theoremstyle{definition}
\theoremstyle{definition}
\theoremstyle{definition}\newtheorem{remark}[theorem]{Remark}
\theoremstyle{definition}
\theoremstyle{definition}

\numberwithin{equation}{section}

\newcommand{\e}{\operatorname{e}}
\newcommand{\N}{\mathbb{N}}
\newcommand{\R}{\mathbb{R}}

\newcommand{\Q}{\mathbb{Q}}
\newcommand{\D}{{\bf D}}											

\newcommand*\diff{\mathop{}\!\mathrm{d}}								
\newcommand{\E}{\mathbb{E}}											
\newcommand{\Pro}{\mathbb{P}}										
\newcommand{\oland}{\mbox{\scriptsize $\mathcal{O}$}}						
\newcommand{\Oland}{\mathcal{O}}										
\newcommand{\ind}{\operatorname{\mathbf{1}}}							
\renewcommand{\sup}{\operatorname{sup}}					                
\renewcommand{\inf}{\operatorname{inf}}			                        
    
\makeatletter
    \def\@fnsymbol#1{\ensuremath{\ifcase#1\or \ddagger\or  \mathsection \or \dagger \or \mathparagraph\or \|\or **\or \dagger\dagger
        \or \ddagger\ddagger \else\@ctrerr\fi}}
\makeatother

\newcommand*\samethanks[1][\value{footnote}]{\footnotemark[#1]}

\title{An Aldous-type tightness characterisation for the $M_1$ topology: replacing global modulus control by local control at stopping times}
\date{\today}

\author{Grégoire Szymanski\thanks{Department of Mathematics, University of Luxembourg} \, and  Fabrice Wunderlich\samethanks \,
\thanks{Department of Statistics, London School of Economics}}

\begin{document}
\maketitle

\begin{abstract}

Aldous' criterion provides a stopping-time characterisation of tightness in Skorokhod's $J_1$ topology and is one of the standard tools for proving functional limit theorems. For the weaker $M_1$ topology, where limiting jumps can be approximated by clustered jumps or steep continuous motion, no analogous characterisation has been available and tightness classically requires controlling a global oscillation modulus over all nearby triples of times. We prove such an Aldous-type characterisation of $M_1$ tightness, replacing the global modulus condition by local estimates at stopping times ruling out rapid reversals after a large excursion. Together with the standard compact containment and boundary conditions, these estimates are both necessary and sufficient for $M_1$ tightness. As a consequence, we derive tractable and general sufficient conditions for additive functionals of strong Markov processes. They reduce the verification of $M_1$ tightness to short-time estimates for the processes restarted from its state at suitable stopping times.

\medskip
\noindent\textbf{Keywords:} Skorokhod $M_1$ topology; tightness; Aldous' criterion; stopping times; functional limit theorems; additive functionals of Markov processes

\smallskip
\noindent\textbf{MSC2020:} Primary 60F17; Secondary 60B10, 60G40, 60J25
\end{abstract}

\section{Introduction}

On the space of càdlàg paths, Skorokhod's $M_1$ topology, which was introduced in 1956 in \cite{skorokhod}, has become an increasingly important mode of functional convergence for stochastic processes with jumps. In contrast to the stronger $J_1$ topology, $M_1$ allows  limiting jumps to be approximated by several clustered jumps or by an increasingly steep continuous motion. Consequently, $M_1$ arises naturally in heavy-tailed and stable limit theorems \cite{avram, basrak_krizmanic_segers, Melbourne_Varandas, Jung_Zhang}, continuous-time and trapped random walks \cite{arous_btm, croydonM1}, queueing and stochastic-network models \cite{Mikosch, PangWhitt, ramanan_M1}, interacting-particle and mean-field systems \cite{Dembo_LiCheng, mercy, Nadtochiy_Skholni2}, and singular stochastic control \cite{Asaf_cohen_M1, Fu_Horst}. We refer to \cite[Ch.~11.5ff.]{whitt} for a detailed treatment of the $M_1$ topology and its role in weak convergence theory, and to \cite[Ch.~3]{billingsley} for the corresponding theory in the stronger $J_1$ topology. \\

While the weaker nature of the $M_1$ topology allows for a much broader class of approximations, establishing tightness and weak convergence often remains substantially less tractable than for its $J_1$ counterpart. In the $J_1$ setting, tightness can often be verified through local estimates, in particular through stopping-time criteria. Such methods include the conditioned-shift approach of Kurtz \cite{kurtz1975}, Jacod, Mémin and Métivier \cite{JacodMeminMetivier1983}, and most prominently Aldous' stopping-time criterion \cite{aldous} together with the related criteria for martingales and semimartingales developed by Rebolledo \cite{rebolledo}. These results reduce tightness to the control of local increments over short random intervals. No corresponding stopping time characterisation has been available for $M_1$.

While rapid increments are compatible with $M_1$ tightness and may in fact approximate a jump of the limit, the relevant obstruction is instead a rapid reversal following a recent large excursion. Our main result shows that controlling precisely such reversals at stopping times yields a necessary and sufficient characterisation of $M_1$ tightness.

The classical characterisation of $M_1$ tightness is formulated through a global modulus which measures the distance of the path from line segments connecting nearby points. More precisely, at its core, one has to verify that, for every $\varepsilon>0$,
\begin{align}
\lim_{\delta\downarrow0}\, 
\limsup_{n\to\infty} \;
\Pro\bigl(w(X^n,\delta)>\varepsilon\bigr) \, = \, 0, \label{eq:classical_tightness_condition_modulus}
\end{align}
where
$$
w(x,\delta)
:= \, 
\sup_{0\le t\le T}\, 
\sup_{(t-\delta)\vee0\le t_1\le t_2\le t_3\le(t+\delta)\wedge T} \ 
{\rm dist}\bigl(x(t_2),[x(t_1),x(t_3)]\bigr)
$$
and 
$${\rm dist}(a,[b,c]):=\operatorname{min}\{|a-\lambda b-(1-\lambda)c|:0\le \lambda \le 1\}.$$
Together with standard compact containment and boundary conditions, Condition \eqref{eq:classical_tightness_condition_modulus} is in fact necessary and sufficient for $M_1$ tightness (see Theorem~\ref{thm:classical_characterisation_tightness_M1} below). This is a \emph{global} geometric condition: one must control the behaviour of the process uniformly over all nearby triples of times. As a consequence, even when the model admits local information, an additional argument is generally required to translate this local information into the required global modulus estimate. 

For the direct verification of $M_1$ tightness, there are currently only a few general mechanisms for avoiding a direct control of the global modulus. The first is monotonicity, for which the oscillation modulus vanishes entirely and the interior modulus condition is automatic~\cite[Cor.~12.5.1]{whitt}. A second mechanism consists of composition arguments based on continuity results for the composition map in the $M_1$ topology, most notably the theorems given in \cite[Ch.~13.2]{whitt}. These allow one to avoid proving $M_1$ tightness directly, provided that tightness and convergence of the driving process and the clock are already available and the limiting pair satisfies suitable structural assumptions. Outside these two settings, the general localisation principle is a sufficient condition given by Avram and Taqqu \cite[Prop.~3]{avram}, based on a Kolmogorov-type estimate for deterministic three-point oscillations:
\begin{equation} 
    \Pro\big({\rm dist}(X^n_{t_2},[X^n_{t_1},X^n_{t_3}])>\varepsilon\big) \; \le \; L\varepsilon^{-\nu}(t_3-t_1)^{1+\beta}. \label{eq:avram_taqqu_Kolmogorov_control}
\end{equation}
Here $L,\beta,\nu>0$ are fixed and the bound is required to hold for all $t_1\le t_2\le t_3$, $\varepsilon>0$, and $n\ge1$; crucially, it is superlinear in $t_3-t_1$. Their criterion constitutes the key localisation step in the proof of their seminal work on weak $M_1$ convergence for heavy-tailed moving averages. In general the condition is, however, rather restrictive: it requires uniform quantitative control of three-point oscillations over arbitrary deterministic triples of times, with a superlinear dependence on the length of the surrounding time interval. Such an estimate can be extremely difficult to establish and need not hold in models with less explicit structure. Unlike \eqref{eq:classical_tightness_condition_modulus}, the estimate \eqref{eq:avram_taqqu_Kolmogorov_control} is only a sufficient condition: it implies \eqref{eq:classical_tightness_condition_modulus}, but it is not necessary for $M_1$ tightness.\\

\noindent \textbf{Contribution of this paper.}
The purpose of the present paper is to provide a stopping-time localisation principle for $M_1$ tightness. We show that the classical modulus condition \eqref{eq:classical_tightness_condition_modulus} admits an equivalent formulation entirely in terms of local estimates at stopping times, in the spirit of Aldous' criterion for $J_1$. Thus, the global modulus condition underlying $M_1$ tightness can be localised without loss of information.

The resulting characterisation identifies rapid reversals following recent excursions as the relevant obstruction to $M_1$ tightness. As a consequence, for strong Markov processes, the required stopping-time estimates reduce by the strong Markov property to short-time estimates after restarting the processes at selected stopping times. We formulate this as a general sufficient criterion for Markov-driven additive functionals and illustrate its use on a chemical reaction network with a discontinuous scaling limit.

\section{Main result} \label{sec:main_results}

Fix $T>0$ and denote the space of all real-valued càdlàg paths $x:[0,T] \to \R$ which are continuous at time $T$ by $\D_{\R}[0,T]$. Let $(X^n)_{n\ge1}$ be a sequence of real-valued adapted càdlàg processes taking values in $\D_{\R}[0,T]$, defined on filtered probability spaces $(\Omega^n,\mathcal F^n,\mathbb F^n,\Pro^n)$ which satisfy the usual conditions. Denote by $\mathcal T^n$ the collection of $\mathbb F^n$-stopping times taking values in $[0,T]$. Throughout, for a càdlàg process $X$ we use the notation $\|X\|_\infty:=\sup_{0\le s\le T}|X_s|$ for the supremum of the absolute value of $X$ on $[0,T]$. Moreover, we will write $a\wedge b:= \operatorname{min}\{a,b\}$ and $a\vee b:= \operatorname{max}\{a,b\}$. The Borel $\sigma$-algebra of $(\D_{\R}[0,T],M_1)$ coincides with the $\sigma$-algebra generated by the coordinate maps $x\mapsto x(t)$, $0\le t\le T$ (see, e.g., \cite[Thm.~11.5.2]{whitt}). Hence each $X^n$ is a random element of $(\D_{\R}[0,T],M_1)$.

We first recall the standard characterisation of $M_1$ tightness, which can be found, for instance, in \cite[Thm.~12.12.3]{whitt}.

\begin{theorem}[Classical characterisation of $M_1$ tightness] \label{thm:classical_characterisation_tightness_M1}
The sequence $(X^n)_{n\ge 1}$ is tight on the space $(\D_{\R}[0,T], M_1)$ if and only if 
\begin{enumerate}[(i)]
\item  \label{it1:classical_charact_tightness_M1} for every $\eta>0$ there exists $R>0$ such that
\begin{equation*}
\limsup_{n\to \infty} \; \mathbb{P}^n\big (\|X^n\|_\infty\geq R\big) \; \le \;  \eta;
\end{equation*}
\item  \label{it2:classical_charact_tightness_M1} for every $\varepsilon>0$,
\begin{equation*}
\lim_{\delta\downarrow0}\ \limsup_{n\to\infty}
\ \mathbb P^n \Big(
\sup_{0\le t\le\delta}
\big(|X_t^n-X_0^n|\vee|X_{T}^n-X_{T-t}^n|\big)>\varepsilon
\Big)\, = \, 0.
\end{equation*}
\item \label{it3:classical_charact_tightness_M1} Condition \eqref{eq:classical_tightness_condition_modulus} is satisfied.
\end{enumerate}
\end{theorem}
For a recent geometric characterisation of $M_1$ tightness for random compact sets of càdlàg real-valued paths, see Freeman and Swart \cite{FreemanSwart2025}.

The philosophy behind our main result, Theorem \ref{thm:tightness_criterion_M1} below, is simple. The oscillation modulus $w(x,\delta)$ can become large only if the path first undergoes a substantial excursion in one direction and then rapidly reverses direction. Consequently, rather than controlling all nearby triples of times simultaneously, it should be sufficient to control the probability of such local reversals.

Theorem \ref{thm:tightness_criterion_M1} below replaces the global requirement in \eqref{eq:classical_tightness_condition_modulus} by local stopping-time estimates. For $\delta>0$ and stopping times $\tau,\sigma\in\mathcal T^n$, define the local reversal functional
\[
\mathcal R_\delta^n(\tau,\sigma)
:= \; 
\sup_{(\tau-\delta)\vee0\le s\le\tau} \ 
\operatorname{dist}\bigl(
X_\tau^n,[X_s^n,X_\sigma^n]
\bigr).
\]
The following is the $M_1$ counterpart of Aldous' stopping-time characterisation of $J_1$ tightness.
\begin{theorem}[Aldous-type characterisation of $M_1$ tightness]\label{thm:tightness_criterion_M1}
The sequence $(X^n)_{n\ge 1}$ is tight on the space $(\D_{\R}[0,T], M_1)$ if and only if Conditions \eqref{it1:classical_charact_tightness_M1} and \eqref{it2:classical_charact_tightness_M1} of Theorem \ref{thm:classical_characterisation_tightness_M1} are satisfied,
and, for every $\varepsilon>0$,
\begin{equation}
\lim_{\delta\downarrow0}\ \limsup_{n\to\infty}
\ \sup_{\substack{\tau, \sigma\in\mathcal T^n\\
\tau\le\sigma\le(\tau+\delta)\wedge T}}
\mathbb P^n\!\left(
\mathcal R_\delta^n(\tau,\sigma)>\varepsilon
\right)
\, = \, 0. \label{it3:main_thm}
\end{equation}
\end{theorem}

By Theorems~\ref{thm:classical_characterisation_tightness_M1} and~\ref{thm:tightness_criterion_M1}, for sequences satisfying Conditions~\eqref{it1:classical_charact_tightness_M1} and~\eqref{it2:classical_charact_tightness_M1}, the local reversal condition~\eqref{it3:main_thm} and the modulus condition~\eqref{eq:classical_tightness_condition_modulus} are equivalent, and each of them is necessary and sufficient for $M_1$ tightness. The Avram--Taqqu estimate~\eqref{eq:avram_taqqu_Kolmogorov_control}, by comparison, implies~\eqref{eq:classical_tightness_condition_modulus} but is not necessary for tightness, as observed in the introduction.

As a local stopping time condition, \eqref{it3:main_thm} of Theorem \ref{thm:tightness_criterion_M1} plays a role analogous to Aldous' criterion for the $J_1$ topology. The crucial distinction is that the relevant object is no longer a future increment but rather a future reversal following a recent excursion. For fixed $n\ge1$, $\delta,\varepsilon>0$, and stopping times
$\tau,\sigma\in\mathcal T^n$ with
$\tau\le\sigma\le(\tau+\delta)\wedge T$, the one-dimensional interval geometry gives
\[
\begin{aligned}
\big\{\mathcal R_\delta^n(\tau,\sigma)>\varepsilon\big\}
\, = \; &\big\{
X_\tau^n-X_\sigma^n>\varepsilon,\;\;
\sup_{(\tau-\delta)\vee0\le s\le\tau}
(X_\tau^n-X_s^n)>\varepsilon
\big\}
\\
&\qquad\cup
\big\{
X_\sigma^n-X_\tau^n>\varepsilon,\;\;
\sup_{(\tau-\delta)\vee0\le s\le\tau}
(X_s^n-X_\tau^n)>\varepsilon
\big\}.
\end{aligned}
\]
Indeed, for each $s$ the distance of $X_\tau^n$ from the interval with endpoints
$X_s^n$ and $X_\sigma^n$ exceeds $\varepsilon$ if and only if
$X_\tau^n$ lies more than $\varepsilon$ above both endpoints or more than
$\varepsilon$ below both endpoints. Consequently, \eqref{it3:main_thm} is equivalent to requiring, for every $\varepsilon>0$ and $\xi\in\{+1,-1\}$,
\begin{equation*}
\begin{aligned}
&\lim_{\delta\downarrow0}\ \limsup_{n\to\infty} \, 
\sup_{\substack{\tau, \sigma\in\mathcal T^n\\
\tau\le\sigma\le(\tau+\delta)\wedge T}}
\mathbb P^n\!\left(
\xi\bigl(X_\tau^n-X_\sigma^n\bigr)>\varepsilon,\;\;
\sup_{(\tau-\delta)\vee0\le s\le\tau}
\xi\bigl(X_\tau^n-X_s^n\bigr)>\varepsilon
\right)
\, = \, 0.
\end{aligned}
\end{equation*}
With a view toward practical application, we emphasise a further equivalent formulation of Theorem \ref{thm:tightness_criterion_M1} in Corollary \ref{cor:main} below which, by restricting the class of stopping times $\tau$, only requires to evaluate the probability of the simple difference between $X_\tau$ and $X_\sigma$. For $\delta,\varepsilon>0$, and $\xi \in \{+1,-1\}$ set
\begin{align}
\mathcal{T}^{n,\xi}_{\delta,\varepsilon}
&:=\Big\{\tau\in\mathcal T^n:
\{\tau<T\}\subseteq
\big\{
\sup_{(\tau-\delta)\vee0\le s\le\tau}
\xi (X_\tau^n-X_s^n)>\varepsilon
\big\}{ \ \ \Pro^n\text{-a.s.}}\Big\}.\label{defi:restricted_stopping_time_classes}
\end{align}
Then we have the following restricted-class reformulation of Theorem~\ref{thm:tightness_criterion_M1}.

\begin{cor}
    \label{cor:main}
The sequence $(X^n)_{n\ge1}$ is tight on $(\D_{\R}[0,T],M_1)$ if and only if
Conditions \eqref{it1:classical_charact_tightness_M1} and
\eqref{it2:classical_charact_tightness_M1} of Theorem~\ref{thm:classical_characterisation_tightness_M1} are satisfied and, for every $\varepsilon>0$ and $\xi \in \{+1,-1\}$,
    \begin{equation*}
        \lim_{\delta\downarrow0}\ \limsup_{n\to\infty}
\ \sup_{\substack{\tau\in\mathcal T^{n,\xi}_{\delta,\varepsilon}, \; \sigma\in\mathcal T^n\\
\tau\le\sigma\le(\tau+\delta)\wedge T}}
\ \mathbb P^n\!\left(
\xi(X_\tau^n-X_\sigma^n)>\varepsilon
\right)
\, = \, 0.
    \end{equation*}
\end{cor}

Consider now $d\ge2$. For processes with values in $\D_{\R^d}[0,T]$, two $M_1$ topologies are commonly distinguished \cite[Sect.~12.3]{whitt}: the weak topology $WM_1$ and the strong topology $SM_1$. Tightness in $WM_1$, where the topology corresponds to the $d$-fold product topology $(\D_{\R}[0,T],M_1)^d$, is equivalent to tightness of the individual coordinate processes, and we therefore immediately obtain the following characterisation.

\begin{cor}
    Let $(X^n)_{n\ge1}$ be $\R^d$-valued adapted càdlàg processes. Then $(X^n)_{n\ge1}$ is tight on $(\D_{\R^d}[0,T],WM_1)$ if and only if each of its components satisfies the conditions of Theorem \ref{thm:tightness_criterion_M1} or Corollary \ref{cor:main}.
\end{cor}

For the strong topology $SM_1$, coordinatewise tightness is not sufficient. Instead, tightness is characterised through the tightness of all one-dimensional projections. Writing $\mathbb S^{d-1}:=\{\xi\in\R^d:|\xi|=1\}$ 
and applying Theorem \ref{thm:tightness_criterion_M1} and Corollary \ref{cor:main} to the projected sequence of processes $(\xi^\top X^n)_{n\ge 1}$, $\xi\in\mathbb S^{d-1}$, yields the following characterisation \cite[Thm.~12.7.2]{whitt}.  

\begin{cor}\label{prop:multidim}
Let $(X^n)_{n\ge1}$ be $\R^d$-valued adapted càdlàg processes. Then, the following are equivalent:
\begin{enumerate}[(a)]
\item \label{it:multidim_a} $(X^n)_{n\ge1}$ is tight on $(\D_{\R^d}[0,T],SM_1)$;
\item \label{it:multidim_c} Conditions \eqref{it1:classical_charact_tightness_M1} and \eqref{it2:classical_charact_tightness_M1} of Theorem~\ref{thm:classical_characterisation_tightness_M1} are satisfied, and \eqref{it3:main_thm} holds for every $\varepsilon>0$, $\xi \in \mathbb S^{d-1}$ with $X^n$ replaced by $\xi^\top X^n$;
\item \label{it:multidim_d} Conditions \eqref{it1:classical_charact_tightness_M1} and \eqref{it2:classical_charact_tightness_M1} of Theorem~\ref{thm:classical_characterisation_tightness_M1} are satisfied and, for every $\varepsilon>0$ and $\xi\in\mathbb S^{d-1}$,
\begin{equation*}
\lim_{\delta\downarrow0}\ \limsup_{n\to\infty}
\ \sup_{\substack{\tau\in\mathcal T^{n,\xi}_{\delta,\varepsilon}, \; \sigma\in\mathcal T^n\\
\tau\le\sigma\le(\tau+\delta)\wedge T}}
\ \mathbb P^n\!\big(
\xi^\top(X_\tau^n-X_\sigma^n)>\varepsilon
\big)
\, = \, 0,
\end{equation*}
where $\mathcal T^{n,\xi}_{\delta,\varepsilon}$ is defined by \eqref{defi:restricted_stopping_time_classes} with $\xi(X^n_\tau-X^n_s)$ replaced by $\xi^\top(X^n_\tau-X^n_s)$.
\end{enumerate}
\end{cor}

\section{Proofs}
\label{sec:proof_theorem}

\subsection{Outline of the proof of Theorem \ref{thm:tightness_criterion_M1}}

We will outline the proof of Theorem \ref{thm:tightness_criterion_M1} by a sequence of lemmas, all of which are proved in Section \ref{sec:proof_lemmas}. 

The difficulty in proving sufficiency is that the times $s<t<u$ realising $w(X^n,\delta)>\varepsilon$ are in general not stopping times: whether $X^n$ falls back after $t$ depends on the future of the path. We therefore observe $X^n$ along a grid of stopping times. Each grid time is the first exit of $X^n$ from the closed interval centred at the previous grid value, with radius of a fixed fraction of $\varepsilon$. By construction, $X^n$ stays within this radius of the current grid value before reaching the following grid time, and consecutive grid values differ by at least this radius. Consequently, if $X^n$ rises and then falls by more than $\varepsilon$ within a short window, the grid values rise and then fall by more than $\varepsilon$ minus twice the radius. Consider the grid time associated to the first downward step after this peak. There, $X^n$ lies more than $\varepsilon$ minus the radius above its value at the start of the window. At the next grid time, still inside the window, it is lower by at least the radius. This is a reversal at a stopping time, of exactly the type controlled by~\eqref{it3:main_thm}. Since~\eqref{it3:main_thm} controls one stopping time at a time, it remains to bound the number of grid times before $T$. A bounded path cannot make too many consecutive grid steps in the same direction, so the grid reverses within every sufficiently long run of steps. Only a bounded number of disjoint slow reversals fit into $[0,T]$, so a large number of grid times forces many fast reversals, which~\eqref{it3:main_thm} makes unlikely.

We now turn to the formal argument. To begin with, Lemma \ref{lem:stopped_paths_random_windows} ensures the measurability of key quantities required later on.

\begin{lem}[Measurability of stopped values and random-window functionals]\label{lem:stopped_paths_random_windows}
Let $X$ be a real-valued adapted càdlàg process on some filtered probability space $(\Omega,\mathcal F,\mathbb F,\Pro)$ satisfying the usual conditions, let $\rho,\sigma,\tau$ be stopping times bounded by $T$, and let $\delta>0$. Let
\[
\mathscr D:=\sigma\{\pi_t:0\le t\le T\},\qquad \pi_t(x):=x(t),
\]
be the canonical coordinate $\sigma$-algebra on $\D_{\R}[0,T]$.
\begin{enumerate}[(a)]
\item $\rho$ and $X_\rho$ are $\mathcal F_\rho$-measurable, and therefore $\mathcal F_T$-measurable. Moreover, the stopped process $X^\rho:= X_{\rho\wedge \cdot}$ is $\mathcal F_\rho/\mathscr D$-measurable as a map $\Omega\to\D_{\R}[0,T]$.
\item The random-window excursion variables
\[
A_\delta(\tau):=\sup_{(\tau-\delta)\vee0\le s\le\tau} \, (X_\tau-X_s),
\qquad
B_\delta(\tau):=\sup_{(\tau-\delta)\vee0\le s\le\tau} \, (X_s-X_\tau)
\]
are $\mathcal F_\tau$-measurable, and
\[
\mathcal R_\delta(\tau,\sigma):=\sup_{(\tau-\delta)\vee0\le s\le\tau} \
\operatorname{dist}\bigl(X_\tau,[X_s,X_\sigma]\bigr)
\]
is $\mathcal F_{\tau \vee \sigma}$-measurable.
\item The maps $x\mapsto\|x\|_\infty$ and $x\mapsto w(x,\delta)$ and, if $\delta\le T$, the map $x\mapsto\sup_{0\le t\le\delta}\bigl(|x(t)-x(0)|\vee|x(T)-x(T-t)|\bigr)$ are $\mathscr D$-measurable. 
\end{enumerate}
\end{lem}

The necessity implication is direct. Assume that $(X^n)_{n\ge1}$ is tight in $(\D_{\R}[0,T],M_1)$. By Theorem
\ref{thm:classical_characterisation_tightness_M1}, Conditions
\eqref{it1:classical_charact_tightness_M1}--\eqref{it3:classical_charact_tightness_M1} hold. Fix $\varepsilon>0$, $\delta>0$,
$n\ge1$, and stopping times $\tau,\sigma\in\mathcal T^n$. By Lemma~\ref{lem:stopped_paths_random_windows}(b), $\mathcal R_\delta^n(\tau,\sigma)$ is $\mathcal F_T^n$-measurable. Moreover, pointwise, it clearly holds
\[
\mathcal R_\delta^n(\tau,\sigma)\, \le \, w(X^n,\delta), \qquad \text{ if } \tau\le\sigma\le(\tau+\delta)\wedge T.
\]
Hence it follows
\[
\sup_{\substack{\tau, \sigma\in\mathcal T^n\\ \tau\le\sigma\le(\tau+\delta)\wedge T}} \mathbb P^n\bigl(\mathcal R_\delta^n(\tau,\sigma)>\varepsilon\bigr)
\; \le \; 
\sup_{\substack{\tau, \sigma\in\mathcal T^n\\ \tau\le\sigma\le(\tau+\delta)\wedge T}} \mathbb P^n\bigl(w(X^n,\delta)>\varepsilon\bigr),
\]
and we deduce \eqref{it3:main_thm}.

We now prove the sufficiency implication.
Assume Conditions \eqref{it1:classical_charact_tightness_M1} and
\eqref{it2:classical_charact_tightness_M1}, as well as \eqref{it3:main_thm}. By Theorem
\ref{thm:classical_characterisation_tightness_M1}, it remains to prove
Condition \eqref{it3:classical_charact_tightness_M1}. To parameterise precisely the stopping-time bounds used below, for $\theta>0$ and $c>0$ define
\[
\alpha_n(\theta;c)
:=
\sup_{\substack{\tau, \sigma\in\mathcal T^n\\ \tau\le\sigma\le(\tau+\theta)\wedge T}}
\mathbb P^n\!\left(
\mathcal R_\theta^n(\tau,\sigma)>c
\right).
\]
Condition~\eqref{it3:main_thm} states precisely that, for every fixed $c>0$,
\begin{align}
\lim_{\theta\downarrow0}\ \limsup_{n\to\infty}\ \alpha_n(\theta;c) \, = \, 0. \label{eq:stopping_time_condition}
\end{align}

Heuristically, we will proceed as follows: we will discretise each path by an oscillation grid. On a bounded path, a large number of grid increments eventually forces sign reversals. The stopping-time Condition \eqref{eq:stopping_time_condition} makes reversals occurring within short time windows unlikely, while at the same time sufficiently many slow reversals cannot fit inside $[0,T]$. A final covering argument then controls the $M_1$ modulus.

To formalise this, we fix $\varepsilon>0$ and put
\[
a:=\varepsilon/4,\qquad b:=a/2.
\]
We then introduce the $a$--oscillation grid by $\tau_0^n:=0$ and random times
\begin{align}
\tau_{k+1}^n
:=\inf\big\{t>\tau_k^n \, : \, |X_t^n-X^n_{\tau_k^n}|>a\big\} \wedge T,
\qquad k\ge0, \label{eq:defi_stopping_times}
\end{align}
with $\inf\emptyset=\infty$. By induction on $k$, $\tau^n_k \in \mathcal{T}^n$ for all $k\ge 0$, $n\ge 1$. Indeed, if $\tau^n_k\in\mathcal T^n$, then for $t\in(0,T]$ right-continuity of the paths and Lemma~\ref{lem:stopped_paths_random_windows}(a) give
\[
\{\tau^n_{k+1}<t\}=\bigcup_{q\in\Q\cap(0,t)}\Bigl(\{\tau^n_k<q\}\cap\bigl\{|X^n_q-X^n_{\tau^n_k}|>a\bigr\}\Bigr)\in\mathcal F^n_t,
\]
and right-continuity of $\mathbb F^n$ yields $\{\tau^n_{k+1}\le t\}\in\mathcal F^n_t$ for $t<T$. Moreover, by right-continuity of the paths,
    \[
    \bigl|X^n_{\tau_{k+1}^n}-X^n_{\tau_k^n}\bigr| \, \ge \, a \qquad \text{on } \{\tau^n_{k+1}<T\}.
    \]
For $\theta>0$, define the auxiliary stopping time
\[
\sigma_{k,\theta}^n
:=\tau_{k+1}^n\wedge(\tau_k^n+\theta), \qquad k\ge0, n\ge 1.
\]
Clearly, $\sigma_{k,\theta}^n\in\mathcal T^n$ and
$
\tau_k^n\le\sigma_{k,\theta}^n\le(\tau_k^n+\theta)\wedge T.
$
Lastly, define the event
\[
H_{k,\theta}^n
:=
\Big\{
\mathcal R_\theta^n\bigl(\tau_k^n,\sigma_{k,\theta}^n\bigr)>b
\Big\}.
\]
By Lemma~\ref{lem:stopped_paths_random_windows}(b), $H_{k,\theta}^n\in\mathcal F_T^n$, and the definition of $\alpha_n(\theta;b)$ together with the preceding admissibility relation gives
\begin{align}
\mathbb P^n(H_{k,\theta}^n)\, \le \, \alpha_n(\theta;b). \label{eq:upper_bound_H_k_theta}
\end{align}
The next lemma shows that a bounded path with many $a$-oscillations must exhibit at least one rapid sign reversal.

\begin{lem}
    \label{lem:small:proba}
Let $R>0$, $M\ge 2$, and $\theta>0$ such that $Ma>2R$, and set $q:=M(\lfloor 2T/\theta\rfloor+1)+1$. Then for every $n\ge1$,
\begin{equation}\label{eq:many_grid_points}
\mathbb P^n\bigl(\tau_q^n< T,\ \|X^n\|_\infty\le R\bigr)
\, \le \, 2M\alpha_n(\theta;b).
\end{equation}
\end{lem}

We now prove~\eqref{it3:classical_charact_tightness_M1}. Let $\eta>0$. Choose $R>0$ such that
\begin{align}
\limsup_{n\to\infty} \ \mathbb P^n(\|X^n\|_\infty>R) \,  \le \, \eta/4. \label{eq:key_estimation_1}
\end{align}
Let $M\ge 2$ such that $Ma>2R$. Then by \eqref{eq:stopping_time_condition} choose $\theta_0>0$ such that
\[
\limsup_{n\to\infty} \ \alpha_n(\theta_0;b) \, < \, \frac{\eta}{8M},
\]
and let $q:=M(\lfloor 2T/\theta_0\rfloor+1)+1$ be given by Lemma \ref{lem:small:proba} applied with $\theta=\theta_0$, which gives us
\begin{align}
\limsup_{n\to\infty} \
\mathbb P^n\bigl(\tau_q^n< T,\ \|X^n\|_\infty\le R\bigr)
\, \le \, 2M\limsup_{n\to\infty}\ \alpha_n(\theta_0;b) \label{eq:key_estimation_2}
\, <\, \eta/4.
\end{align}
Further, again by \eqref{eq:stopping_time_condition}, and by Condition \eqref{it2:classical_charact_tightness_M1}, we can choose, for this fixed $q$, $0 < \theta < \theta_0$ such that
\begin{align}
\limsup_{n\to\infty}\ \big[ q\,\alpha_n(\theta;b) + \Pro^n(\sup_{0\le t\le\theta}
|X_{T}^n-X_{T-t}^n| > a)\big] \, < \,\eta/4. \label{eq:key_estimation_3}
\end{align}
This two-scale choice is the crux of the argument: $q$ is fixed through $\theta_0$ before $\theta$ is chosen, and $\theta<\theta_0$ is then taken small enough to compensate for the factor $q$ arising below from the union bound over the events $H^n_{1,\theta},\ldots,H^n_{q-1,\theta}$.

We then require the following final covering.
\begin{lem}
    \label{lemma:last}
Let $\theta>0$, $q\ge 1$ and $R>0$. For every $n\ge1$ and every $\delta\in(0,\theta/2 ]$,
\begin{align}
\{w(&X^n,\delta)>\varepsilon\} \notag \\
&\subseteq \, 
\{\|X^n\|_\infty>R\}
\cup
\{\tau_q^n< T,\ \|X^n\|_\infty\le R\} \cup \{\sup_{0\le t\le\theta}
|X_{T}^n-X_{T-t}^n| > a\} \cup 
\bigcup_{k=1}^{q-1}H_{k,\theta}^n. \label{eq:pathwise_cover}
\end{align}
\end{lem}

By \eqref{eq:upper_bound_H_k_theta},
\[
\sum_{k=1}^{q-1}\mathbb P^n(H_{k,\theta}^n) \, 
\le \, q\alpha_n(\theta;b).
\]
Therefore, for all $0<\delta\le\theta/2$, exploiting \eqref{eq:key_estimation_1}, \eqref{eq:key_estimation_2}, and \eqref{eq:key_estimation_3} gives us
\[
\limsup_{n\to\infty} \ \mathbb P^n(w(X^n,\delta)>\varepsilon)
\, \le \, \eta/4+\eta/4+\eta/4 \,< \, \eta .
\]
Since $\delta\mapsto w(x,\delta)$ is nondecreasing for every $x$, the function
\[
\delta\, \longmapsto \, \limsup_{n\to\infty}\ \mathbb P^n\bigl(w(X^n,\delta)>\varepsilon\bigr)
\]
is nondecreasing. Hence its limit as $\delta\downarrow0$ exists, and the preceding bound, valid for every $\delta\le\theta/2$, shows that this limit is at most $\eta$. Since $\eta>0$ was arbitrary,
\[
\lim_{\delta\downarrow0}\, \limsup_{n\to\infty}
\ \mathbb P^n(w(X^n,\delta)>\varepsilon) \, = \, 0,
\]
which is \eqref{it3:classical_charact_tightness_M1}. We deduce the tightness of $(X^n)_{n\ge 1}$ in $(\D_{\R}[0,T],M_1)$.

\subsection{Proof of the invoked lemmas}
\label{sec:proof_lemmas}

\begin{proof}[Proof of Lemma \ref{lem:stopped_paths_random_windows}]

Let $Q_T:=(\Q\cap[0,T])\cup\{T\}$. Since every $s\in[u,v)$ is the limit of a sequence of rationals in $(s,v)$ decreasing to $s$, every right-continuous $y:[0,T]\to\R$ satisfies, for $0\le u\le v\le T$,
\begin{equation}\label{eq:countable_sup}
\sup_{u\le s\le v}y(s)=\sup\bigl(\{y(v)\}\cup\{y(q):q\in\Q\cap[u,v)\}\bigr).
\end{equation}

(a) Fix $t\in[0,T]$. Since $\{\rho\wedge t\le r\} = \{\rho\le r\}\in\mathcal F_r\subseteq\mathcal F_t$
for $r<t$ and $\{\rho\wedge t\le r\}  = \Omega$ for $r\ge t$, the random variable $\rho\wedge t$ is $\mathcal F_t$-measurable. Hence, for every $m\ge1$, $\rho_m:=\bigl(2^{-m}\lceil 2^m(\rho\wedge t)\rceil\bigr)\wedge t$ is $\mathcal F_t$-measurable with values in the finite set 
$$V_m:=\{k2^{-m}:k\in\N\cup\{0\},\ k2^{-m}<t\}\cup\{t\},$$
so that $X_{\rho_m}=\sum_{v\in V_m}X_v\mathbf 1_{\{\rho_m=v\}}$ is $\mathcal F_t$-measurable. As $\rho\wedge t\le\rho_m\le\rho\wedge t+2^{-m}$, right-continuity of the paths gives $X_{\rho_m}\to X_{\rho\wedge t}$, so $X_{\rho\wedge t}$ is $\mathcal F_t$-measurable. Since $\rho=\rho\wedge t$ on $\{\rho\le t\}$, for every Borel set $B\subseteq\R\times[0,T]$,
\[
\{(X_\rho,\rho)\in B\}\cap\{\rho\le t\}=\{(X_{\rho\wedge t},\rho\wedge t)\in B\}\cap\{\rho\le t\}\in\mathcal F_t.
\]
As $\{\rho\le t\}=\{\rho\le T\}$ for $t\ge T$, $X_\rho$ and $\rho$ are $\mathcal F_\rho$-measurable.\\
Now, the stopped path $X^\rho(t)=X_{t\wedge\rho(\omega)}(\omega)$, $0\le t\le T$, is càdlàg, and its coordinates $X_{t\wedge\rho}$ are $\mathcal F_{t\wedge\rho}$-measurable and hence $\mathcal F_\rho$-measurable. Since $\mathscr D$ is generated by the coordinate maps, $X^\rho$ is $\mathcal F_\rho/\mathscr D$-measurable. \\

(b) For $y\in\D_{\R}[0,T]$ and $u\in[0,T]$, \eqref{eq:countable_sup} with $v=T$ gives
\[
S(y,u):=\sup_{u\le s\le T}y(s)=\sup_{q\in Q_T}\bigl(y(q)\mathbf 1_{\{u\le q\}}+y(T)\mathbf 1_{\{u>q\}}\bigr).
\]
Since $Q_T$ is countable, $S$ is $\mathscr D\otimes\mathcal B([0,T])$-measurable. Let $L:=(\tau-\delta)\vee0\le\tau$. Since $X^\tau_s=X_s$ for $s\le\tau$ and $X^\tau_s=X_\tau$ for $s\ge\tau$,
\[
A_\delta(\tau)=X^\tau(T)+S(-X^\tau,L),\qquad B_\delta(\tau)=S(X^\tau,L)-X^\tau(T),
\]
and these are $\mathcal F_\tau$-measurable because, making use of (a), $\omega\mapsto\bigl(\pm X^\tau(\omega),L(\omega)\bigr)$ is $\mathcal F_\tau/(\mathscr D\otimes\mathcal B([0,T]))$-measurable. Next, for $a,b,c\in\R$ the segment $[b,c]$ is the interval $[b\wedge c,b\vee c]$, whence
\begin{equation}\label{eq:dist_formula}
\operatorname{dist}(a,[b,c])
=(a-(b\vee c))^+\vee((b\wedge c)-a)^+
=\bigl((a-b)^+\wedge(a-c)^+\bigr)\vee\bigl((b-a)^+\wedge(c-a)^+\bigr).
\end{equation}
We take $a=X_\tau$, $b=X_s$, $c=X_\sigma$ and the supremum over $s\in[L,\tau]$. Since $z\mapsto z^+$ and $z\mapsto z\wedge\alpha$ are continuous and nondecreasing, and $A_\delta(\tau),B_\delta(\tau)\ge0$, we obtain
\[
\mathcal R_\delta(\tau,\sigma)
=\bigl(A_\delta(\tau)\wedge(X_\tau-X_\sigma)^+\bigr)\vee\bigl(B_\delta(\tau)\wedge(X_\sigma-X_\tau)^+\bigr),
\]
which is $\mathcal F_{\tau \vee \sigma}$-measurable by (a).\\

(c) By \eqref{eq:countable_sup}, applied to $|x|$ on $[0,T]$, to $|x(\cdot)-x(0)|$ on $[0,\delta]$ and to $|x(T)-x(\cdot)|$ on $[T-\delta,T]$, the first and the third map in (c) are countable suprema of $\mathscr D$-measurable maps. Next, times $0\le t_1\le t_2\le t_3\le T$ lie in a common window $[(t-\delta)\vee0,(t+\delta)\wedge T]$, $t\in[0,T]$, if and only if $t_3-t_1\le2\delta$, and $d_x(t_1,t_2,t_3):=\operatorname{dist}\bigl(x(t_2),[x(t_1),x(t_3)]\bigr)\ge0$ vanishes if $t_2\in\{t_1,t_3\}$. Let $\Delta$ be the countable, non-empty set of triples $(q_1,q_2,q_3)\in Q_T^3$ with $q_1<q_2<q_3$ and $q_3-q_1\le2\delta$. Each map $x\mapsto d_x(q_1,q_2,q_3)$ is $\mathscr D$-measurable by \eqref{eq:dist_formula}. Hence it suffices to show $d_x(t_1,t_2,t_3)\le\sup_\Delta d_x$ whenever $t_1<t_2<t_3$ and $t_3-t_1\le2\delta$, since then $w(x,\delta)=\sup_\Delta d_x$. For $k\ge 1$ with $2/k<(t_2-t_1)\wedge(t_3-t_2)$, choose $q_3^k\in Q_T\cap[t_3,t_3+1/k)$, then $q_1^k\in\Q\cap(t_1+h_k,t_1+h_k+1/k)$, where $h_k:=q_3^k-t_3$, and $q_2^k\in\Q\cap(t_2,t_2+1/k)$. Then $t_1<q_1^k<t_2<q_2^k<t_3\le q_3^k$ and $q_3^k-q_1^k<t_3-t_1\le2\delta$, so that $(q_1^k,q_2^k,q_3^k)\in\Delta$; moreover, $q_i^k\ge t_i$ and $q_i^k\to t_i$. By right-continuity of $x$ and continuity of \eqref{eq:dist_formula} in $(a,b,c)$, $d_x(q_1^k,q_2^k,q_3^k)\to d_x(t_1,t_2,t_3)$.
\end{proof}

\begin{proof}[Proof of Lemma \ref{lem:small:proba}]
    Take $B = \lfloor 2T/\theta\rfloor+1$, and set $q:=BM+1$. On $\{\tau_q^n< T,\ \|X^n\|_\infty\le R\}$, each block
\[
\tau_{jM}^n,\tau_{jM+1}^n,\ldots,\tau_{(j+1)M}^n,
\qquad j=0,\ldots,B-1,
\]
contains at least one sign reversal among consecutive grid increments. Indeed, if all
$M$ increments in the block had the same sign, the value of the path would
change by at least $Ma>2R$, which is impossible on $\{\lVert X^n\rVert_\infty \le R\}$. 
For $k\ge1$, define
\[
\Delta_k^n:=X^n_{\tau_k^n}-X^n_{\tau_{k-1}^n}.
\]
Let $F_j^n$ be the event that, for some
$k\in\{jM+1,\ldots,(j+1)M-1\}$,
\[
\tau_{k+1}^n< T,\qquad
\tau_{k+1}^n-\tau_{k-1}^n\le\theta,
\qquad
\Delta_k^n\Delta_{k+1}^n<0.
\]
This event is $\mathcal F_T^n$-measurable by Lemma~\ref{lem:stopped_paths_random_windows}(a). On $\{\tau_{k+1}^n< T\}$, both increments $\Delta_k^n$ and $\Delta_{k+1}^n$ have absolute value at least $a$. Thus $\Delta_k^n\Delta_{k+1}^n<0$ means exactly that the two consecutive increments have opposite signs, and $X^n_{\tau_k^n}$ lies at distance at least $a$ from the interval with endpoints $X^n_{\tau_{k-1}^n}$ and $X^n_{\tau_{k+1}^n}$. If $F_j^n$ occurs through an index $k$, then obviously $\sigma_{k,\theta}^n=\tau_{k+1}^n$ and
$\tau^n_k-\tau_{k-1}^n< \theta$. Since $a>b$, this implies that $H_{k,\theta}^n$ occurs. Hence
\[
F_j^n \, \subseteq \, \bigcup_{k=jM+1}^{(j+1)M-1}H_{k,\theta}^n,
\]
and therefore
\[
\mathbb P^n(F_j^n) \, \le \,M \alpha_n(\theta;b)
\]
by \eqref{eq:upper_bound_H_k_theta}.

Let $E_n:=\{\tau_q^n< T,\ \|X^n\|_\infty\le R\}$. On $E_n$, if $F_j^n$ does not occur, then the $j$-th block contains a reversal between some grid times $\tau^n_{k-1}, \tau^n_k$, and $\tau^n_{k+1}$, but every such reversal satisfies $\tau_{k+1}^n-\tau_{k-1}^n>\theta$. Since $\tau_{jM}^n\le \tau_{k-1}^n< \tau_{k+1}^n\le\tau_{(j+1)M}^n$ on $E_n$, the duration of that block is therefore greater than $\theta$. The $B$ block intervals are consecutive and disjoint apart from their endpoints, and on $E_n$ their total duration is
\[
\sum_{j=0}^{B-1}\bigl(\tau_{(j+1)M}^n-\tau_{jM}^n\bigr)
\, = \,\tau_{BM}^n \, \le \, \tau_q^n \, < \, T.
\]
Hence fewer than $T/\theta$ blocks can have duration greater than $\theta$. Since $B\theta>2T$, one has $T/\theta<B/2$, and thus on $E_n$ more than $B/2$ of the events $F_j^n$ must occur. We deduce that, pointwise, it holds
\[
\frac B2\,\mathbf1_{E_n}
\, \le \, \sum_{j=0}^{B-1}\mathbf1_{F_j^n}.
\]
Taking expectations and using the preceding bound on $\mathbb P^n(F_j^n)$ gives
\[
\mathbb P^n(E_n)
\, \le \, \frac2B\sum_{j=0}^{B-1}\mathbb P^n(F_j^n)
\, \le \, 2M\alpha_n(\theta;b),
\]
which proves~\eqref{eq:many_grid_points}.
\end{proof}

\begin{proof}[Proof of Lemma \ref{lemma:last}]

To prove the claim, let us consider the events
\begin{align*}
\big( \{ \|X^n\|_\infty> &R\} \, \cup \, \{\tau_q^n<T, \|X^n\|_\infty\le R\} \, \cup \, \{\sup_{0\le t\le\theta}
|X_{T}^n-X_{T-t}^n| > a\} \big)^\complement \\
&= \; \{\tau_q^n\ge T, \, \|X^n\|_\infty\le R, \, \sup_{0\le t\le\theta}
|X_{T}^n-X_{T-t}^n| \le a\}.
\end{align*}
On this event, suppose that $w(X^n,\delta)>\varepsilon$. By the definition of $w$, there exist 
$(u-2\delta)\vee 0 \le s<t<u\le T$ such that $\operatorname{dist}\bigl(X_t^n,[X_s^n,X_u^n]\bigr)>\varepsilon$. Since $\delta\le\theta/2$, we have $u-s\le2\delta\le\theta$. Moreover, it must hold that $u<T$. Indeed, otherwise $\varepsilon<|X^n_u-X^n_t|\le \sup_{0\le t\le\theta}
|X_{T}^n-X_{T-t}^n| \le a \le \varepsilon/4$ which leads to a contradiction. Now, by replacing $X^n$ with $-X^n$ if necessary, it is enough to treat the case
\[
X_t^n-X_s^n>\varepsilon,\qquad X_t^n-X_u^n>\varepsilon.
\]
For $r\in[0,T)$, set 
$$i(r):=\max\{k:\tau_k^n\le r\}.$$ 
This index exists and is smaller than $q$ because $\tau_0^n=0$ and $\tau_q^n\ge T> r$. By maximality, $r<\tau_{i(r)+1}^n$, and hence
$
\bigl|X_r^n-X^n_{\tau_{i(r)}^n}\bigr|\le a.
$
Set
\[
j:=i(s),\qquad m:=i(t),\qquad \ell:=i(u).
\]
Since $j\le m\le\ell$ and
\[
X_{\tau_m^n}^n-X_{\tau_j^n}^n
>\varepsilon-2a=\varepsilon/2>0,
\qquad
X_{\tau_m^n}^n-X_{\tau_\ell^n}^n
>\varepsilon-2a=\varepsilon/2>0,
\]
we have $j<m<\ell$. In particular, $X_{\tau_m^n}^n>X_{\tau_\ell^n}^n$, implying that not all grid increments starting with the increment from the $m$-th to the $(m+1)$-th grid time up until the increment from the $(\ell-1)$-th to the $\ell$-th grid time can be positive. Let $r_0$ be the first index in $\{m+1,\ldots,\ell\}$ with
\[
X_{\tau_{r_0}^n}^n-X_{\tau_{(r_0-1)}^n}^n \, < \, 0,
\]
and put $k:=r_0-1$. Since $\tau^n_{r_0}\le \tau^n_{\ell}< T$, by definition of the stopping times \eqref{eq:defi_stopping_times},
\[
X_{\tau_k^n}^n-X_{\tau_{k+1}^n}^n \, \ge \, a \, > \, b.
\]
All preceding increments starting from the $m$-th grid time through the $k$-th grid time are positive, so
$X_{\tau_k^n}^n\ge X_{\tau_m^n}^n$. Therefore
\[
X_{\tau_k^n}^n-X_s^n
\, \ge \, X_{\tau_m^n}^n-X_s^n
\, > \, \varepsilon-a \, = \, \frac{3\varepsilon}{4}>b.
\]
Because $j<m$, maximality of $j=i(s)$ gives $s<\tau_m^n$. Thus
\[
s <\tau_m^n\le\tau_k^n<\tau_{k+1}^n\le\tau_\ell^n\le u.
\]
Consequently
\[
s\in[(\tau_k^n-\theta)\vee0,\tau_k^n],
\qquad
\tau_{k+1}^n-\tau_k^n\le u-s\le\theta,
\]
so $\sigma_{k,\theta}^n=\tau_{k+1}^n$. Moreover, $j<m\le k<\ell<q$, and hence $1\le k\le q-1$. The preceding inequalities therefore imply
\[
\operatorname{dist}\bigl(
X^n_{\tau_k^n},[X_s^n,X^n_{\tau_{k+1}^n}]
\bigr)>b, \quad\text{ and hence }\quad
\mathcal R_\theta^n(\tau_k^n,\sigma_{k,\theta}^n)>b.
\]
Thus $H_{k,\theta}^n$ occurs.

In the opposite directional case, apply the same argument to $-X^n$. The oscillation grid is unchanged because it is defined through absolute increments, and interval distance is invariant under multiplication by $-1$, hence the same event $H_{k,\theta}^n$ occurs. This proves~\eqref{eq:pathwise_cover}.
\end{proof}

\subsection{Proof of Corollary \ref{cor:main}}

By Theorem~\ref{thm:tightness_criterion_M1}, it suffices to show that \eqref{it3:main_thm} is equivalent to the two restricted-class conditions in the statement. Fix $n\ge1$, $\delta,\varepsilon>0$, and $\xi\in\{+1,-1\}$. For
$\tau\in\mathcal T^n$, define
\[
A_\xi^n(\tau)
:=\big\{
\sup_{(\tau-\delta)\vee0\le s\le\tau}
\xi\bigl(X_\tau^n-X_s^n\bigr)>\varepsilon
\big\}.
\]
By Lemma~\ref{lem:stopped_paths_random_windows}(b), $A_\xi^n(\tau)\in\mathcal F_\tau^n$. Recall that  \eqref{it3:main_thm} is equivalent to 
\[
\lim_{\delta\downarrow0}\ \limsup_{n\to\infty}
\ \sup_{\substack{\tau, \sigma\in\mathcal T^n\\
\tau\le\sigma\le(\tau+\delta)\wedge T}} \, 
\mathbb P^n\!\left(
\xi\bigl(X_\tau^n-X_\sigma^n\bigr)>\varepsilon,\;
A_\xi^n(\tau)
\right) \, = \, 0,
\]
for both choices of $\xi$.

We are going to identify the inner supremum with the corresponding restricted-class supremum. First, let $\xi=1$, let $\tau\in\mathcal T^{n,\xi}_{\delta,\varepsilon}$, and let $\sigma$ be admissible. On $\{X_\tau^n-X_\sigma^n>\varepsilon\}$ one necessarily has $\tau<T$; hence, by the definition of $\mathcal T^{n,\xi}_{\delta,\varepsilon}$, $A_1^n(\tau)$ occurs $\Pro^n$-almost surely. Therefore
\[
\mathbb P^n\!\left(X_\tau^n-X_\sigma^n>\varepsilon\right)
\, = \, 
\mathbb P^n\!\left(X_\tau^n-X_\sigma^n>\varepsilon,\;A_1^n(\tau)\right).
\]
Conversely, take arbitrary admissible $\tau,\sigma\in\mathcal T^n$, $\tau\le \sigma \le (\tau+\delta)\wedge T$, and set $A:=A_1^n(\tau)$. Define $\widehat\tau:=\tau\mathbf1_A+T\mathbf1_{A^\complement}$ and $\widehat\sigma:=\sigma\mathbf1_A+T\mathbf1_{A^\complement}$. Since $A\in\mathcal F_\tau^n\subseteq\mathcal F_\sigma^n$, both random times are stopping times: for $t<T$,
$\{\widehat\tau\le t\}=A\cap\{\tau\le t\}$ and $\{\widehat\sigma\le t\}=A\cap\{\sigma\le t\}$.
Moreover,
\[
\widehat\tau\le\widehat\sigma\le(\widehat\tau+\delta)\wedge T,
\qquad
\widehat\tau\in \mathcal T^{n,+1}_{\delta,\varepsilon},
\qquad
\text{and}\qquad
\left\{X_\tau^n-X_\sigma^n>\varepsilon,\;A\right\}
=
\left\{X_{\widehat\tau}^n-X_{\widehat\sigma}^n>\varepsilon\right\},
\]
from which we easily conclude that both suprema coincide, proving the claim for $\xi=+1$.

For $\xi=-1$ the same argument applies, and the conclusion follows from Theorem~\ref{thm:tightness_criterion_M1}.

\subsection{Proof of Corollary \ref{prop:multidim}}

The proof essentially follows from \cite[Thm.~12.7.2]{whitt}, combined with Theorem~\ref{thm:tightness_criterion_M1} and Corollary~\ref{cor:main} applied to $(\xi^\top X^n)_{n\ge1}$, $\xi\in\mathbb S^{d-1}$. 
The only additional ingredient is the following approximation of the sphere $\mathbb S^{d-1}$ by a finite $\rho$-net. For $\R^d$-valued paths $x$, one has $w(x,\delta)=\sup_{\xi\in\mathbb S^{d-1}}w(\xi^\top x,\delta)$, and $\xi\mapsto w(\xi^\top x,\delta)$ is Lipschitz continuous with constant $2\|x\|_\infty$. Hence, for every finite $\rho$-net $N_\rho$ of $\mathbb S^{d-1}$,
\[
w(x,\delta)\,\le\,\max_{\xi\in N_\rho}w(\xi^\top x,\delta)+2\rho\,\|x\|_\infty .
\]
Together with Condition \eqref{it1:classical_charact_tightness_M1}, this reduces \eqref{eq:classical_tightness_condition_modulus} for $(X^n)_{n\ge1}$ to \eqref{eq:classical_tightness_condition_modulus} for the finitely many projections $(\xi^\top X^n)_{n\ge1}$, $\xi\in N_\rho$, with $\rho>0$ chosen small enough.

\section{Application to Markov-driven additive processes}
\subsection{From moving averages to Markovian dependence}
A major and highly influential line of development in $M_1$ limit theory arose from the study of heavy-tailed dependent sequences, where the verification of tightness has historically been a central challenge. The functional $M_1$ theory in this setting was initiated by Avram and Taqqu \cite{avram}, who established weak functional $M_1$ convergence for moving averages driven by heavy-tailed i.i.d. innovations $\zeta_k$, that is for processes of the form
\begin{align} X^n_t := \; a_n^{-1}\sum_{k=0}^{\lfloor nt \rfloor} Z_k, \qquad Z_k  := \; \sum_{j=0}^\infty c_j \zeta_{k-j}, \qquad \qquad t\ge 0,\,  n\ge 1, \label{eq:moving_averages_defi}
\end{align}
where $c_j\ge 0$ for all $j\ge 0$ and $(a_n)_{n\ge 1}$ is a suitable normalising sequence, under additional regularity assumptions. Their work opened a substantial line of research concerned with stable limits under increasingly general forms of dependence.

A further step toward more general dependence structures was taken by Bartkiewicz, Jakubowski, Mikosch and Wintenberger \cite{bartkiewicz_jakubowski_mikosch_wintenberger}, who established stable limit theorems for the normalised partial sums of strictly stationary dependent and heavy-tailed sequences, at the level of endpoint convergence. Basrak, Krizmanić and Segers \cite{basrak_krizmanic_segers} subsequently developed this theory at the functional level, proving weak functional convergence in the $M_1$ topology for general strictly stationary, jointly regularly varying sequences satisfying suitable anti-clustering and asymptotic independence conditions.

A particularly important class within this broader dependence framework arises when the dependence is generated by an underlying Markov process. Jara, Komorowski and Olla \cite{jara_komorowski_olla} established stable endpoint limits for additive functionals of discrete-time Markov chains under suitable ergodicity and tail assumptions and, for an integrated functional of a continuous-time pure-jump Markov process, obtained convergence of finite-dimensional distributions to those of a stable process. Mikosch and Wintenberger \cite{mikosch_wintenberger} established endpoint limits for partial sums of regularly varying functions of stationary discrete-time Markov chains under drift conditions of Foster--Lyapunov type. Cattiaux and Manou-Abi \cite{cattiaux_manou_abi} subsequently pursued the functional theory for additive functionals of discrete-time Markov chains by deriving suitable mixing properties from strong ergodic assumptions and then applying existing stable functional limit theorems for stationary mixing sequences. Their approach relies on strong ergodic assumptions in order to reduce the problem to existing limit theorems for stationary sequences.

Our Theorem~\ref{thm:tightness_MDAP} below should be viewed as complementing this line of work at the level of $M_1$ tightness. Rather than imposing assumptions designed to produce a particular stable limit such as stationarity, regular variation, mixing, or drift conditions, we exploit only the strong Markov structure needed for the tightness argument itself. The underlying process and its state space may vary with $n$, and the result applies to general (pathwise) additive functionals of general continuous-time strong Markov processes; partial sums of discrete-time Markov chains are covered by Remark~\ref{rem:discrete_time} below. Consequently, whenever finite-dimensional convergence is available from model-specific arguments, Theorem~\ref{thm:tightness_MDAP} provides a local mechanism for upgrading it to functional convergence in $M_1$ (see \cite[Thm.~11.6.6]{whitt}).\\

For each $n\ge 1$, let $E^n$ be a metric space with Borel $\sigma$-algebra $\mathcal E^n$, and equip the space $\D_{E^n}[0,\infty)$ of càdlàg paths $y:[0,\infty)\to E^n$ with the $\sigma$-algebra $\mathscr D^n:=\sigma\{y\mapsto y(t):t\ge0\}$. For $t\ge0$, let $\theta_t$ denote the shift on $\D_{E^n}[0,\infty)$, $(\theta_ty)(s):=y(t+s)$, $s\ge0$. Let $(\Pro^n_x)_{x\in E^n}$ be a family of probability measures on $(\D_{E^n}[0,\infty),\mathscr D^n)$ such that $\Pro^n_x(\{y:y(0)=x\})=1$ for every $x\in E^n$ and $x\mapsto\Pro^n_x(B)$ is $\mathcal E^n$-measurable for every $B\in\mathscr D^n$. Further, let $(\Phi^n_t)_{t\ge 0}$ be a càdlàg $E^n$-valued process, defined on a filtered probability space $(\Omega^n,\mathcal F^n,\mathbb F^n,\Pro^n)$, which satisfies the usual conditions. And let $\Phi^n$ be $\mathbb F^n$-adapted and strong Markov with respect to $\mathbb F^n$ and with transition family $(\Pro^n_x)_{x\in E^n}$, that is, for every $\mathbb F^n$-stopping time $\tau$ and every $B\in\mathscr D^n$ it holds that
\begin{equation}\label{eq:strong_Markov}
\Pro^n\bigl(\theta_\tau\Phi^n\in B \, | \, \mathcal{F}^n_\tau \bigr)
\; = \; \Pro^n_{\Phi^n_\tau}(B) \qquad \text{ on } \{\tau < \infty\}.
\end{equation}
Here, by the argument in the proof of Lemma~\ref{lem:stopped_paths_random_windows}(a), $\Phi^n_\tau$ is $\mathcal F^n_\tau/\mathcal E^n$-measurable and $\theta_\tau\Phi^n$ is $\mathcal F^n/\mathscr D^n$-measurable on $\{\tau<\infty\}$. Further, let $A^n:\D_{E^n} [0,\infty) \to \D_\R [0,\infty)$ be such that $y\mapsto A^n_t(y):=A^n(y)(t)$ is $\mathscr D^n$-measurable for every $t\ge0$, such that the process $(A^n_t(\Phi^n))_{t\ge0}$ is $\mathbb F^n$-adapted, and such that, for some $G^n\in\mathscr D^n$ with $\Pro^n(\Phi^n\in G^n)=1$, the additivity property
\begin{align}
A^n_{t+s}(y) \; = \; A^n_t(y) + A^n_s(\theta_t y), \qquad y\in G^n,\ s,t\ge 0, \label{eq:additivity_property}
\end{align}
holds.
For a sequence $(a_n)_{n\ge 1} \subseteq (0,\infty)$, define
$$ X^n_t := \; a_n^{-1} A^n_t(\Phi^n), \qquad t\in [0,T].$$
We will call $X^n$ a \emph{Markov-driven additive process} (MDAP). In accordance with the standing assumption of Section~\ref{sec:main_results}, we assume that $\Delta X^n_T=0$ $\Pro^n$-almost surely.

A simple example of an MDAP generating additive path functional $A^n$ is
$$ A^n_t(y) \, = \, \int_0^t f_n(y(s))\diff s \, + \, \sum_{0<u\le t} g_n(y(u-),y(u)), \qquad  t\ge0,$$
and $A^n(y)\equiv0$,
where $f_n:E^n\to \R$, $g_n:E^n\times E^n\to \R$ can be any measurable functions with $g_n(x,x)=0$ for all $x\in E^n$ (so that the sum runs over the jump times of $y$), such that, for all $t\ge0$,
$$ \int_0^t |f_n(y(s))|\diff s < \infty \qquad \text{ and } \qquad \sum_{0<u\le t}| g_n(y(u-),y(u))| < \infty.$$
Indeed, the path $t\mapsto A^n_t(y)$ is càdlàg, and \eqref{eq:additivity_property} holds. Moreover, $A^n_t(\Phi^n)$ coincides with a functional of $(\Phi^n_s)_{s\le t}$, so that $(A^n_t(\Phi^n))_{t\ge0}$ is $\mathbb F^n$-adapted.

\begin{theorem}[Tightness conditions for MDAPs] \label{thm:tightness_MDAP}
Let $(X^n)_{n\ge1}$ be MDAPs, and suppose that for all $n\ge1$, $\delta,\varepsilon>0$, and all $\xi\in\{+1,-1\}$,
there exist sets $K^\xi_{n,\delta,\varepsilon}\in\mathcal E^n$, such that the following conditions hold:
\begin{enumerate}
    \item the uniform stochastic boundedness and endpoint conditions \eqref{it1:classical_charact_tightness_M1}, \eqref{it2:classical_charact_tightness_M1} from Theorem \ref{thm:classical_characterisation_tightness_M1} are satisfied; \label{it1:tightness_MDAP}
    \item for every $\varepsilon>0$ and $\xi \in \{+1,-1\}$, \label{it2:tightness_MDAP}
    \begin{align*}
        \limsup\limits_{\delta \downarrow 0}\ \limsup_{n\to \infty} \ \sup_{\tau \in \mathcal{T}^{n,\xi}_{\delta,\varepsilon}} \ \Pro^n \big( \tau < T, \, \Phi^n_{\tau } \notin K^\xi_{n,\delta,\varepsilon}\big) \; = \; 0;
    \end{align*}
    \item for every $\varepsilon>0$ and $\xi \in \{+1,-1\}$, \label{it3:tightness_MDAP}
    \begin{align*}
        \limsup\limits_{\delta \downarrow 0}\ \limsup_{n\to \infty} \ \sup_{x \in K^\xi_{n,\delta, \varepsilon}} \ \Pro^n_x \Big( \inf_{0\le t \le \delta}\,  \xi A_t^n \, < \, -\varepsilon a_n\Big) \; = \; 0,
    \end{align*}
\end{enumerate}
where we use the convention $\operatorname{sup}(\emptyset):=0$ and where, under $\Pro^n_x$, $A^n_t$ is evaluated at the canonical process $Y_t(y):=y(t)$. Then the sequence $(X^n)_{n\ge 1}$ is tight on $(\D_\R[0,T], M_1)$.
\end{theorem}
Note that, implicitly, we require that the sets of stopping times $\mathcal{T}^{n,\xi}_{\delta, \varepsilon}$ are defined with respect to the filtrations $\mathbbm{F}^n$ for which the $X^n$ are MDAPs.
\begin{proof}[Proof of Theorem \ref{thm:tightness_MDAP}]
    First, let $\delta,\varepsilon>0$, $n\ge 1$, $\tau \in \mathcal{T}^{n,+1}_{\delta, \varepsilon}$, and $\sigma\in\mathcal T^n$ with $\tau \le \sigma \le (\tau+\delta)\wedge T $. Then, on the event $\{\Phi^n\in G^n\}$ of full $\Pro^n$-probability,
    $$ X^n_\tau - X^n_\sigma \; = \; - a_n^{-1}\big(A^n_\sigma(\Phi^n) - A^n_\tau(\Phi^n)\big) \; = \;-a_n^{-1} A^n_{\sigma-\tau}(\theta_\tau \Phi^n),$$
    by \eqref{eq:additivity_property} with $t=\tau$ and $s=\sigma-\tau\in[0,\delta]$, and therefore, up to a $\Pro^n$-null set,
   \begin{align} \{X^n_\tau - X^n_\sigma > \varepsilon\} \; \subseteq \; \big \{ \inf_{0\le t\le \delta} \, A^n_{t}(\theta_\tau \Phi^n) \, < \, - \varepsilon a_n \big\} . \label{eq:proof_application_thm_1}\end{align}
    The event on the right-hand side of \eqref{eq:proof_application_thm_1} is $\{\theta_\tau\Phi^n\in B\}$, where $B:=\{y\in\D_{E^n}[0,\infty):\inf_{0\le t\le\delta}A^n_t(y)<-\varepsilon a_n\}$ belongs to $\mathscr D^n$, since by right-continuity the infimum may be restricted to $t\in(\Q\cap[0,\delta])\cup\{\delta\}$. Moreover, $\sigma=T$ on $\{\tau=T\}$, so that $\{X^n_\tau - X^n_\sigma > \varepsilon\}\subseteq\{\tau<T\}$. Since $\{\Phi^n_\tau\in K^+_{n,\delta,\varepsilon}\}\in\mathcal F^n_\tau$, the tower property yields
    \begin{align}
        \Pro^n (X^n_\tau - X^n_\sigma > \varepsilon) \; &\le \; \Pro^n \big( \tau < T, \, \Phi^n_{\tau} \notin K^+_{n,\delta,\varepsilon}\big) \, + \, \E^n \Big[ \ind_{\{\Phi^n_{\tau} \in K^+_{n,\delta,\varepsilon}\}} \Pro^n (X^n_\tau -X^n_\sigma > \varepsilon \, | \, \mathcal{F}^n_\tau )\Big]. \label{eq:proof_application_thm_2}
    \end{align}
    By \eqref{eq:proof_application_thm_1}, the second summand on the right-hand side is bounded by 
    \begin{align*}
        \E^n \Big[ \ind_{\{\Phi^n_{\tau} \in K^+_{n,\delta,\varepsilon}\}}\,  \Pro^n \big(\inf_{0\le t\le \delta} \, A^n_{t}(\theta_\tau \Phi^n) \, < \, - \varepsilon a_n \, \big | \, \mathcal{F}^n_\tau \big)\Big].
    \end{align*}
    The strong Markov property \eqref{eq:strong_Markov} implies equality of this term to
   \begin{align*}
   \E^n \Big[ \ind_{\{\Phi^n_{\tau} \in  K^+_{n,\delta,\varepsilon}\}}&\,  \Pro^n_{\Phi^n_{\tau}} \big(B^+ \, \big)\Big]
   \le \; \sup_{x \in K^+_{n,\delta, \varepsilon}} \ \Pro^n_x \big( B^+\big),
   \end{align*}
    where the supremum is $0$ if $K^+_{n,\delta,\varepsilon}=\emptyset$. The resulting bound to \eqref{eq:proof_application_thm_2} together with the assumptions of the theorem, then yield the conditions of Corollary~\ref{cor:main} for $\xi=+1$. The proof for the case $\xi=-1$ follows in full analogy.
\end{proof}

\begin{remark}[Discrete-time chains]\label{rem:discrete_time}
Let $(Y^n_k)_{k\ge0}$ be a time-homogeneous Markov chain taking values in $(E^n,\mathcal E^n)$ and defined on $(\Omega^n,\mathcal F^n,\Pro^n)$. For $x\in E^n$ denote by $\Pro_x^n$ the laws of the chain when started from $x$, let $f_n:E^n\to\R$ be measurable and define
$$
X^n_t:=a_n^{-1}\sum_{k=1}^{\lfloor nt\rfloor}f_n(Y^n_k), \qquad t\in[0,T].
$$
Here $\Delta X^n_T=a_n^{-1}f_n(Y^n_{nT})$ if $nT\in\N$, and $\Delta X^n_T=0$ otherwise. Hence the standing assumption $\Delta X^n_T=0$ of Section~\ref{sec:main_results} holds if and only if, for every $n\ge1$, either $nT\notin\N$ or $f_n(Y^n_{nT})=0$ $\Pro^n$-a.s.; this is the case, e.g., if $T\notin\Q$. Without loss of generality, assume that $T \notin \Q$ throughout this remark, otherwise extend the $X^n$ continuously by $X^n_{T+t}=X^n_T$.
Let $\mathcal F^n_t$ be the $\Pro^n$-completion of $\sigma(Y^n_0,\ldots,Y^n_{\lfloor nt\rfloor})$, and let it be constant on each interval $[k/n,(k+1)/n)$. The filtrations $(\mathcal{F}^n_t)_{0\le t\le T}$ then satisfy the usual conditions. For $\tau\in\mathcal T^n$, the discrete random time $\lfloor n\tau\rfloor$ is a stopping time for the discrete filtration $(\mathcal F^n_{k/n})_{k\ge0}$, and if $\tau\le\sigma\le(\tau+\delta)\wedge T$, then 
$$X^n_\tau-X^n_\sigma=-a_n^{-1}\sum_{k=\lfloor n\tau\rfloor+1}^{\lfloor n\sigma\rfloor}f_n(Y^n_k)$$
with $0\le\lfloor n\sigma\rfloor-\lfloor n\tau\rfloor\le\lfloor n\delta\rfloor+1$. Therefore the proof of Theorem~\ref{thm:tightness_MDAP}, with the discrete-time strong Markov property at $\lfloor n\tau\rfloor$ in place of \eqref{eq:strong_Markov}, shows that $(X^n)_{n\ge1}$ is tight on $(\D_\R[0,T],M_1)$ under Condition~\ref{it1:tightness_MDAP}, Condition~\ref{it2:tightness_MDAP} with $\Phi^n_\tau$ replaced by $Y^n_{\lfloor n\tau\rfloor}$, and
\[
\limsup_{\delta\downarrow0}\ \limsup_{n\to\infty}\ \sup_{x\in K^\xi_{n,\delta,\varepsilon}}\Pro^n_x\Bigl(\min_{0\le j\le \lfloor n\delta\rfloor+1}\ \xi\sum_{k=1}^{j}f_n(Y^n_k)<-\varepsilon a_n\Bigr)=0
\]
for every $\varepsilon>0$ and $\xi\in\{+1,-1\}$.
\end{remark}


\subsection{A chemical reaction network with a discontinuous limit}
\label{sec:application_CRN}

As a natural and brief illustration of the benefits of Theorem~\ref{thm:tightness_MDAP}, consider the
four-species chemical reaction network studied by Laurence and Robert
\cite{laurence_robert_aimd}. Let
$$\Phi^N=(X^{1,N},X^{2,N},X^{3,N},X^{4,N})$$
denote their càdlàg strong Markov jump process with state space $\N_0^4$, where $\N_0:=\{0, 1, 2, \dots\}$, which represents the existing numbers of particles of four species, and let $(\Pro^N_x)_{x\in\N_0^4}$ be its transition family. In this section, we write $\Pro$, $\E$ and $\Pro_x$ for $\Pro^N$, $\E^N$ and $\Pro^N_x$. We focus on the
fourth coordinate, whose only possible transitions are
$$
x_4\longmapsto x_4+1
\quad\text{at rate }\kappa_3x_2x_3,
\qquad
x_4\longmapsto x_4-1
\quad\text{at rate }\kappa_5x_4,
$$
for constants $\kappa_3,\kappa_5>0$. For example, the probability of an increase in the number of particles of species four by one particle during an infinitesimal time interval $(t,t+h]$ is $\kappa_3x_2x_3 h$, i.e.
$$ \mathbb P\big(\Phi^{N}_{t+h}=(x_1,x_2,x_3,x_4+1) \mid \Phi^{N}_t=(x_1,x_2,x_3,x_4) \big) \; = \;
\kappa_3 x_2 x_3 \,h+\oland(h).$$
{As in \cite[Equation~(40)]{laurence_robert_aimd}, $\Phi^N$ can be realised as the solution of a system of stochastic differential equations driven by independent Poisson point processes $\mathcal P_0,\ldots,\mathcal P_5$ on $\R_+^2$ with intensity $\diff s\otimes\diff u$, where $\mathcal P_i$ drives the reaction with rate constant $\kappa_i$; in particular,
\begin{equation}\label{eq:CRN_SDE_4}
X^{4,N}_t=X^{4,N}_0+\int_{(0,t]\times\R_+}\ind_{\{u\le\kappa_3X^{2,N}_{s-}X^{3,N}_{s-}\}}\,\mathcal P_3(\diff s,\diff u)-\int_{(0,t]\times\R_+}\ind_{\{u\le\kappa_5X^{4,N}_{s-}\}}\,\mathcal P_5(\diff s,\diff u).
\end{equation}
The chain is non-explosive and the paths of $\Phi^N$ are piecewise constant with finitely many jumps on compact intervals, and, $\Phi^N$ satisfies \eqref{eq:strong_Markov} with respect to the usual augmentation $\mathbb F^N$ of its natural filtration.

We consider the initial condition
$$\Phi^N_0=(0,N,1,v_N), \qquad \frac{v_N}{\sqrt N}\longrightarrow v\ge0,$$
that is, the setting of \cite[Sect.~6.1]{laurence_robert_aimd} with $y_N=N$.
For this initial condition, Laurence and Robert \cite[Prop.~16]{laurence_robert_aimd} show that $(X^{4,N}_t/\sqrt N)_{t\in[0,T]}$ converges in distribution, in the
$M_1$ topology, to a discontinuous AIMD process with state space $\R_+$, that is a process which exhibits an exponential decay between jumps or a negative jump proportional
to the current value of the state (multiplicative decrease) and, for the increasing
part, positive jumps depending on the current state or a linear growth (additive
increase). Their approach to proving $M_1$ convergence is to remove the fast reaction intervals, establish $J_1$ convergence for the resulting time-changed process, and then recover $M_1$ convergence on the original time scale by linearly interpolating the microscopic staircases on the fast intervals and comparing the resulting paths in the $M_1$ metric. With respect to establishing tightness, our criterion avoids such pathwise reconstructions: after a macroscopic upward movement, it suffices to rule out a rapid reversal. Since all negative jumps of $X^{4,N}$ are deaths occurring at total rate $\kappa_5X^{4,N}$, this reduces to a local short-time death estimate.\\

Define 
$$A_t^N(Y):=Y^4_t-Y^4_0,\qquad a_N:=\sqrt N.$$
Then $A^N$ is additive and \eqref{eq:additivity_property} holds with $G^N=\D_{\N_0^4}[0,T]$. Moreover, the process
$$ Z_t^N:=\frac{A_t^N(\Phi^N)}{\sqrt N} = \frac{X^{4,N}_t-X^{4,N}_0}{\sqrt N}$$
is an MDAP. Since $v_N/\sqrt N\to v$, tightness of $(Z^N)_{N\ge 1}$ is equivalent
to tightness of $(X^{4,N}/\sqrt N)_{N\ge 1}$. The uniform stochastic boundedness and endpoint conditions in Condition \ref{it1:tightness_MDAP} of Theorem \ref{thm:tightness_MDAP} follow readily from the cycle decomposition and short-time estimates established in \cite{laurence_robert_aimd}. In particular, only finitely many fast cycles occur on compact time intervals, each produces an $\Oland_{\Pro}(1)$ contribution after scaling by $\sqrt N$, and outside these cycles the fourth coordinate evolves only through deaths. It remains to show Conditions \ref{it2:tightness_MDAP} and \ref{it3:tightness_MDAP} of Theorem \ref{thm:tightness_MDAP}.

Fix $\varepsilon,\delta >0$ and define
$$K^+_{N,\delta,\varepsilon} := \{x\in\N_0^4:x_4\le \sqrt{N/\delta} \}.$$
Uniform stochastic boundedness of the sequence $(X^{4,N}/\sqrt{N})_{N\ge 1}$ immediately gives Condition \ref{it2:tightness_MDAP} for $\xi=+1$.
It remains to control a rapid reversal after such a positive excursion.
If $\Phi_0^N=x=(x_1,x_2,x_3,x_4)$, let $D_x(t)$ be the number of initial particles of species four which have
disappeared by time $t$. Their remaining number forms a pure-death chain with rate $\kappa_5 \times \emph{\text{its previous value}}$, whose transition probabilities solve classical Kolmogorov equations. From there, it is straightforward to show that
$$ D_x(\delta) \sim \operatorname{Bin}\big(x_4,1-e^{-\kappa_5\delta}\big).
$$
Particles created after time zero cannot contribute negatively relative to
the initial level. Hence, pathwise,
$$ A_t^N(\Phi^N)\, \ge \, -D_x(t), \qquad t\ge0.$$
Consequently, since $D_x$ is nondecreasing, by Markov's inequality and the fact that $1-\exp(-\kappa_5 \delta) \le \kappa_5 \delta$,
\begin{align*}
\sup_{x\in K^+_{N,\delta,\varepsilon}} \Pro_x\big( \inf_{0\le t\le\delta}A_t^N(\Phi^N) <-\varepsilon\sqrt N \big) \; \le  \; \sup_{x\in K^+_{N,\delta,\varepsilon}} \frac{\E\big[ D_x(\delta) \big]}{\varepsilon \sqrt{N}} \; \le \; \frac{\kappa_5\sqrt\delta}{\varepsilon} \; \longrightarrow \; 0
\end{align*}
as $\delta \to 0$, which proves Condition \ref{it3:tightness_MDAP} for $\xi=+1$.

For the opposite direction, we are simply going to show that 
\begin{align} \lim_{\delta\downarrow0}\, \limsup_{N\to\infty} \, \sup_{\tau\in\mathcal T^{N,-1}_{\delta,\varepsilon}} \, \Pro(\tau<T) \; = \;0, \label{eq2:proof_CRN} \end{align}
from which both Conditions \ref{it2:tightness_MDAP} and \ref{it3:tightness_MDAP} follow for $\xi=-1$ with $K^-_{N,\delta,\varepsilon}=\emptyset$. By definition \eqref{defi:restricted_stopping_time_classes} of $\mathcal T^{N,-1}_{\delta,\varepsilon}$, for every $\tau\in\mathcal T^{N,-1}_{\delta,\varepsilon}$ and up to a $\Pro$-null set,
$$\{\tau<T\} \subseteq \Big\{ \sup_{\substack{0\le s\le t\le T\\t-s\le\delta}} \bigl(X^{4,N}_s-X^{4,N}_t\bigr) >\varepsilon\sqrt N \Big\}.$$
It therefore suffices to show that the probability of the event on the right-hand side
vanishes. Let
$$ C^N_t := \sum_{0<u\le t} \ind_{\{\Delta X^{4,N}_u\, =\, -1\}},\qquad \Delta X^{4,N}_u:=X^{4,N}_u-X^{4,N}_{u-},$$
denote the number of deaths of particles of species four up to time $t$. Then
$$ X^{4,N}_s -X^{4,N}_t \, \le \, C^N_t- C^N_s, \qquad 0\le s\le t.$$
Thus a downward motionlarger than $\varepsilon\sqrt N$ requiresmore than $\varepsilon\sqrt N$ death reactions during the same interval. Fix $M>0$ and define
$$\rho^N_M := \inf\Big\{t\ge0:X_t^{4,N}>M\sqrt N\Big\}, \qquad \lambda_N:=\kappa_5M\sqrt N .
$$
For $t\le\rho^N_M$, the transition rate corresponding to a death of a particle of species four is bounded by $
\kappa_5X_{t-}^{4,N}\le \lambda_N$. Now, combining the above
\begin{align} \Pro(\tau< T)  \; \le \;  \Pro(\rho^N_M\le T) +
\Pro\Big( \exists\,k\in\{0,\ldots,\lceil T/\delta \rceil\}: C^N_{(k+2)\delta \wedge\rho^N_M} - C^N_{k\delta\wedge\rho^N_M} \ge \varepsilon \sqrt{N} \Big) \label{eq:1_proof_CRN}
,
\end{align}
since $C^N$ is increasing and any increase of a certain size on an interval of length $\delta$ implies an increase of at least this size between $k\delta$ and $(k+2)\delta$ for some $k\in\{0,\ldots,\lceil T/\delta \rceil\}$. By the Poisson representation \eqref{eq:CRN_SDE_4}, the deaths of species four are generated by the points of the Poisson point process $\mathcal P_5$ on $\mathbb R_+^2$ lying below the level $\kappa_5X_{s-}^{4,N}$. Since before $\rho_M^N$ this level is bounded by $\lambda_N$, the number of such deaths on any interval of length $2\delta$ is stochastically dominated by a $\operatorname{Pois}(2\lambda_N\delta)$ random variable. Hence,
$$  \Pro \Big( C^N_{(k+2)\delta\wedge\rho_M^N} - C^N_{k\delta\wedge\rho_M^N} \ge\varepsilon\sqrt N \Big)\;  \le  \; \Pro\left( \operatorname{Pois}(2\lambda_N\delta) \ge\varepsilon\sqrt N \right). $$
Let $\gamma>0$. First choose $M>0$ large enough such that, by Condition~\eqref{it1:classical_charact_tightness_M1} of Theorem~\ref{thm:classical_characterisation_tightness_M1} and $v_N/\sqrt N\to v$, $\limsup_{N\to\infty}\Pro(\rho^N_M\le T)\le \gamma$, and subsequently choose $\delta>0$ sufficiently small that $2\kappa_5M\delta<\varepsilon$. Then, from \eqref{eq:1_proof_CRN}, applying a union bound and estimating the Poisson tail by a standard Chernoff bound
\begin{align*} 
    \Pro(\tau< T) \; &\le \; \Pro(\rho_M^N\le T) + ( \lceil T/\delta \rceil+1)
\Pro\big( \operatorname{Pois}(2\lambda_N\delta)\ge \varepsilon\sqrt N \big) 
\\
\; &\le \; \Pro(\rho_M^N\le T) + ( \lceil T/\delta \rceil+1) \e^{-c\sqrt N},
\end{align*}
for some constant $c=c(M,\delta,\varepsilon)>0$. Letting $N\to \infty$, and since $\gamma$ was arbitrary, we deduce \eqref{eq2:proof_CRN}.

Theorem~\ref{thm:tightness_MDAP} therefore applies: the sequence $(Z^N)_{N\ge1}$ is tight on $(\D_\R[0,T],M_1)$, and hence so is the sequence $(X^{4,N}/\sqrt N)_{N\ge1}$.

\printbibliography

\end{document}